\documentclass[11pt]{amsart}

\usepackage{amsmath}
\usepackage{amssymb}
\usepackage{graphicx}

\newtheorem{theorem}{Theorem}[section]
\newtheorem{corollary}[theorem]{Corollary}
\newtheorem{lemma}[theorem]{Lemma}
\newtheorem{proposition}[theorem]{Proposition}
\theoremstyle{definition}

\newtheorem{example}[theorem]{Example}

\numberwithin{equation}{section}

\usepackage{mathtools, bbm, hyperref}
\usepackage[shortlabels]{enumitem}
\newcommand{\N}{\mathbb{N}}
\newcommand{\R}{\mathbb{R}}
\newcommand{\C}{\mathbb{C}}
\newcommand{\dmath}{\mathop{}\!\mathrm{d}}
\newcommand{\calG}{\mathcal{G}}
\newcommand{\calH}{\mathcal{H}}
\newcommand{\calK}{\mathcal{K}}
\newcommand{\calL}{\mathcal{L}}
\newcommand{\calM}{\mathcal{M}}
\newcommand{\calO}{\mathcal{O}}

\newcommand{\sgn}{\operatorname{sgn}}
\newcommand{\dom}{\operatorname{dom}}

\usepackage[dvipsnames]{xcolor}

\setlist[enumerate]{topsep=3pt, itemsep=3pt, leftmargin=*}
\setlist[itemize]{topsep=3pt, itemsep=3pt}
\setlist[enumerate,1]{label={\upshape(\roman*)}}

\title[Weak coupling for Schrödinger operators in two dimensions]{Revisiting the Weak Coupling Phenomenon for Two-Dimensional Schrödinger Operators}

\author[J. Behrndt]{Jussi Behrndt}
\author[P. Siegl]{Petr Siegl}
\author[N. Weber]{Nicolas Weber}

\address[J. Behrndt, P. Siegl, N. Weber]{Institut für Angewandte Mathematik, Technische Universität Graz, Steyrergasse 30, 8010 Graz, Austria}
\email{\tt behrndt@tugraz.at, siegl@tugraz.at, nweber@math.tugraz.at}



\begin{document}

\dedicatory{Dedicated to Barry Simon on the occasion of his 80th birthday.}

\begin{abstract}
We study the existence of negative eigenvalues for two-dimensional Schrödinger operators with
real-valued potentials in the weak coupling regime.
In his pioneering paper \cite{Simon1976} from half a century ago, Simon was the first to describe the unique negative eigenvalue
emerging from the threshold of the essential spectrum of one- and two-dimensional Schrödinger operators.
The aim of this paper is to extend Simon's results in two dimensions to a broader class of potentials, allowing for both stronger singularities
and slower decay at infinity, at the cost of losing uniqueness of weakly coupled eigenvalues.
\end{abstract}

\maketitle


\section{Introduction and results}

It is well-known that Schrödinger operators {$-\Delta - V$} in $L^2(\R^d)$, $ d = 1, 2 $, 
have negative eigenvalues for arbitrarily small attractive potentials. More precisely, if $V:\R^d\to\R$ satisfies suitable 
integrability conditions and if the coupling $ \varepsilon > 0 $ is sufficiently small, then $ {-\Delta - \varepsilon V} $ 
has a negative eigenvalue if and only if
\begin{align} \label{simon_U}
    V \not\equiv 0 \quad\text{and}\quad \int_{\R^d} V(x)\dmath x \geq 0.
\end{align}
In this case, the negative eigenvalue is unique and simple. This behaviour is specific to dimensions one and two; in contrast, for $ d \geq 3 $ 
the Cwikel-Lieb-Rozenblum bound guarantees the absence of negative eigenvalues of $ {-\Delta - \varepsilon V} $ for real-valued potentials 
$ V \in L^{d/2}(\R^d) $ and sufficiently small couplings $ \varepsilon > 0 $.

This so-called weak coupling phenomenon was first described by Simon in his celebrated paper \cite{Simon1976}, where
he studied the weakly coupled eigenvalues, including 
existence, absence, uniqueness, asymptotic expansions, as well as their analyticity at zero,
for Schrödinger operators in one and two dimensions.
In $ d = 1 $ he considered potentials satisfying
\begin{align*} 
	(1 + |{\cdot}|^2)V \in L^1(\R)
\end{align*}
and showed that the weakly coupled eigenvalue, if it exists, satisfies
\begin{align*} 
	\sqrt{-\lambda_\varepsilon}
	=
	\frac{\varepsilon}{2} \int_{\R} V(x) \dmath x - \frac{\varepsilon^2}{4} \int_{\R^2} V(x)|x-y|V(y) \dmath(x,y) + o(\varepsilon^2),
	\quad{\varepsilon\to 0+}.
\end{align*}

One year later Klaus \cite{Klaus1977} extended Simon's results \cite{Simon1976} for $ d = 1 $ 
to the broader class of potentials satisfying $ {(1+|{\cdot}|)V \in L^1(\R)} $. The case
\begin{align} \label{long_range}
	V(x) \sim |x|^{-2+\delta},
	\quad |x|\to\infty, \quad\text{for some}\quad \delta\in [0,1)
\end{align}
in one dimension was subsequently studied in \cite{Blankenbecler1977}; for such potentials
a sharp spectral transition occurs at $ \delta = 0 $. If $ V $ satisfies \eqref{simon_U} and
\eqref{long_range} with $ \delta = 0 $ one still finds that $ {-\Delta - \varepsilon V} $ 
has exactly one negative eigenvalue as $ \varepsilon\to 0+ $.
In contrast, if $ V $ satisfies \eqref{long_range} with $ {\delta\in(0,1)} $, 
then $ {-\Delta - \varepsilon V} $ has 
infinitely many negative eigenvalues even if $ \int_{\R} V(x) \dmath x < 0 $.
In both cases, if $ \int_{\R} V(x) \dmath x > 0 $, the ground state admits the expansion
\begin{align*}
	\sqrt{-\lambda_\varepsilon} 
	= \frac{\varepsilon}{2} \int_{\R} V(x) \dmath x + o(\varepsilon), \quad {\varepsilon\to 0+}.
\end{align*}
For $ \delta=0 $, the latter can be further refined by isolating the second term, which is of 
order $ \varepsilon^2 \ln\varepsilon $, and then estimating the new remainder, see \cite{Blankenbecler1977}.

In this paper we focus on the results that Simon obtained 
in two dimensions. We recall them (including well-known spectral properties) in the following theorem.

\begin{theorem}[{\cite[Thm.~3.4]{Simon1976}}] \label{Thm:simon}
Let $ \varepsilon > 0 $ and assume that $ V \in L^1(\R^2) $ is real-valued and satisfies
\begin{equation}\label{simon_eta}
	V\in L^{1+\eta}(\R^2) 
	\quad\text{for some}\quad \eta > 0
\end{equation}
and 
\begin{equation} \label{simon_s}
	|{\cdot}|^t V \in L^1(|x|>1) 
	\quad\text{for some}\quad t > 0.
\end{equation}
Then the following assertions hold.
\begin{enumerate}
	\item The operator $ {H_\varepsilon = -\Delta - \varepsilon V} $ defined in \eqref{H_eps_form}
	is self-adjoint in $ L^2(\R^2) $, bounded from below and one has $ \sigma_{\rm{ess}}(H_\varepsilon) = [0,\infty) $.
	\item If $ V \not\equiv 0 $ and $ \int_{\R^2} V(x) \dmath x \geq 0 $, then $ H_\varepsilon $ has exactly one negative eigenvalue
	$ \lambda_\varepsilon $ for all sufficiently small $ \varepsilon > 0 $, which is simple. 
	If $ \int_{\R^2} V(x) \dmath x > 0 $, then this eigenvalue satisfies
		\begin{align*} 
		\ln(-\lambda_\varepsilon) = 
		-4\pi \left[ \int_{\R^2} V(x) \dmath x \right]^{-1} \varepsilon^{-1} + o(\varepsilon^{-1}), \quad {\varepsilon\to 0+};
	\end{align*}
	if $ \int_{\R^2} V(x) \dmath x = 0 $, then this eigenvalue satisfies
	\begin{align} \label{Simon:U=0}
		\ln(-\lambda_\varepsilon) = -\frac{C}{\varepsilon^2} + o(\varepsilon^{-2}),
		\quad {\varepsilon\to 0+},
	\end{align}
	where $ C > 0 $ is a constant.
	\item If $ \int_{\R^2} V(x) \dmath x < 0 $, then for all sufficiently small $ \varepsilon > 0 $ there are no negative eigenvalues of $ H_\varepsilon $.
\end{enumerate}
\end{theorem}
\smallskip\vspace{-0.2cm}

After Simon's pioneering paper \cite{Simon1976} many works followed that studied limiting absorption principles and
the threshold behaviour of Schrödinger and other differential operators at the edge of the essential spectrum, see, e.g., 
\cite{Simon1977,Blankenbecler1977, Klaus1979, Lakaev1980, Patil1980, Klaus1982, Patil1982,
Holden1985, GesztesyHolden1987, BullaGesztesyRengerSimon1997, FassariKlaus1998}, 
the monograph \cite[Thm.~XIII.11, pp.~336-338]{ReedSimonMethodsIV} and for more recent developments 
\cite{FrankMorozovVugalter2011, KondejLotoreichik2014, ExnerKondejLotoreichik2018, CueninMerz2021, HoangHundertmarkRichterVugalter2023, MolchanovVainberg2023, ExnerKondejLotoreichik2024,Weidl-1999-24}.

However, to the best of our knowledge, there are no results that 
relax the integrability conditions \eqref{simon_eta} and \eqref{simon_s} on $ V $ in the two-dimensional case.
As a small contribution to Simon's work in \cite{Simon1976}
we prove the following theorem.

\begin{theorem} \label{Thm}
Let $ \varepsilon > 0 $ and assume that $ V \in L^1(\R^2) $ is real-valued and satisfies
\begin{equation}\label{V_roll}
 \int_{|x-y|< e} |V(x)|(\ln|x-y|)^2 |V(y)|\dmath(x,y) < \infty
\end{equation}
and 
\begin{equation}\label{V_ln_s}
 |{\ln|{\cdot}||}^s V \in L^1(|x|>1)  
	\quad\text{for some}\quad s\in[0,1).
\end{equation}
Then the following assertions hold.
\begin{enumerate}
	\item The operator $ {H_\varepsilon = -\Delta - \varepsilon V} $ defined in \eqref{H_eps_form}
	is self-adjoint in $ L^2(\R^2) $, bounded from below and one has $ \sigma_{\rm{ess}}(H_\varepsilon) = [0,\infty) $.
	\item If $ \int_{\R^2} V(x) \dmath x > 0 $, then $ H_\varepsilon $ has a negative eigenvalue
	$ \lambda_\varepsilon $ for all sufficiently small $ \varepsilon > 0 $. This eigenvalue is simple and satisfies
		\begin{align}\label{exp_ev} 
			\ln(-\lambda_\varepsilon) = 
			-4\pi \left[ \int_{\R^2} V(x) \dmath x \right]^{-1} \varepsilon^{-1} + o(\varepsilon^{-1+s}), \quad {\varepsilon\to 0+}.
		\end{align}
\end{enumerate}
\end{theorem}

We explain below that our assumptions on $V$ are weaker than those in Simon’s classical work as they allow for stronger local singularities of $V$ as well as a slower decay at infinity. A natural consequence of the latter, however, is the loss of uniqueness of the weakly coupled eigenvalue, see Example~\ref{Ex}(i). 

In more detail, Assumption \eqref{V_roll} is weaker than \eqref{simon_eta}; note that both conditions, roughly speaking, 
control the local singularities of $ V $. Indeed, if $ V \in L^{1+\eta}(\R^2) $ for some $ \eta > 0 $, 
then an application of Fubini's theorem and Hölder's inequality yields
\begin{equation} \label{simon_roll}
\begin{aligned} 
	\int_{|x-y|<e} &|V(x)|(\ln|x-y|)^2 |V(y)| \dmath(x,y) \\
	&=
	\int_{\R^2} |V(x)| \int_{|x-y|<e} (\ln|x-y|)^2 |V(y)| \dmath y \dmath x  \\
	&\leq
	\bigg( \int_{|u|<e} |{\ln|u|}|^{2+\frac{2}{\eta}} \dmath u \bigg)^{\frac{\eta}{1+\eta}} \| V \|_{L^1(\R^2)} \| V \|_{L^{1+\eta}(\R^2)}
	<
	\infty.
\end{aligned}
\end{equation}
Next, it is not difficult to check that Assumption \eqref{simon_s} for some $ t > 0 $ implies Assumption
\eqref{V_ln_s} for any $ s \geq 0 $ and that both conditions are essentially decay restrictions on $ V $. 
However, \eqref{V_ln_s} also covers the case $ s = 0 $, which includes potentials $ V \in L^1(\R^2) $ without any additional decay.
We also note that potentials $ V \in L^1(\R^2) $ satisfying \eqref{V_roll} are relatively form compact perturbations
of $-\Delta$ in $ L^2(\R^2) $, see Corollary~\ref{cor:V_formc}.

\begin{example} \label{Ex}
(i) Consider the potential
\begin{align} \label{Ex_infty}
	V_\infty(x) \coloneqq \frac{\mathbbm{1}_{\lbrace |x|>3 \rbrace}(x)}{|x|^2(\ln|x|)^{1+\delta}(\ln\ln|x|)^2},
	\quad \delta\in[0,1).
\end{align}
Clearly, $ V_\infty \in L^1(\R^2) \cap L^\infty(\R^2) $, but \eqref{simon_s} does not hold for any $ t > 0 $. 
However, it is simple to check that $ V_\infty $ satisfies \eqref{V_ln_s} for $ s=\delta $ and hence Theorem~\ref{Thm} 
implies that for all sufficiently small $ \varepsilon > 0 $ there exists 
a simple negative eigenvalue of $ H_\varepsilon $ that satisfies
\eqref{exp_ev}. At the same time a variational argument (see, e.g., \cite{Chadan2002}) 
shows that $ H_\varepsilon $ has infinitely many negative eigenvalues for any $ \varepsilon > 0 $.

(ii) Next, consider the potential
\begin{align} \label{Ex_0}
	V_0(x) \coloneqq \frac{\mathbbm{1}_{\lbrace |x|<\frac13 \rbrace}(x)}{|x|^2 (\ln|x|)^4}.
\end{align}
It is simple to check that $ V_0 $ belongs to $ L^1(\R^2) $ and satisfies both \eqref{simon_s} and 
\eqref{V_ln_s} for any $ t > 0 $ and $ s \geq 0 $, respectively.
Moreover, $ V_0 $ satisfies \eqref{V_roll} (see Section~\ref{sec:ex} for details) but
$ V_0 \not\in L^{1+\eta}(\R^2) $ for any $ \eta > 0 $, i.e., \eqref{simon_eta} does not hold.
However, Theorem~\ref{Thm} still implies that for every sufficiently small $ \varepsilon > 0 $ there exists a 
simple negative eigenvalue of $ H_\varepsilon $ that satisfies \eqref{exp_ev}.
\end{example}

As a final remark, let us also comment on the situation when the potential $ V $ satisfies
\eqref{V_ln_s} with $ s \geq 1 $. In this case the expansion
\eqref{exp_ev} can be further refined; as in \cite{Weber-2025,KondejLotoreichik2014}, 
the second term can be derived together with corresponding estimates on the remainder. In addition, one can then also
consider potentials of zero mean, i.e., $ \int_{\R^2} V(x) \dmath x = 0 $, for which the existence of a simple bound
state satisfying \eqref{Simon:U=0} can be shown by using arguments analogous to \cite{Simon1976}.

\section{Definition and properties of \texorpdfstring{$-\Delta-V$}{the Schrödinger operator}}  \label{sec:def}

In the following we define a self-adjoint realization $ H $ of $ {-\Delta - V} $ in $ L^2(\R^2) $ assuming that
$ V \in L^1(\R^2) $ is real-valued and satisfies \eqref{V_roll};
note that the condition \eqref{V_ln_s} is not used until Section~\ref{sec:WC}.
Under these assumptions it turns out that the 
Birman-Schwinger operator
\begin{align} \label{Q}
	Q(\alpha)
	\coloneqq
	\overline{
	|V|^\frac12 (-\Delta + \alpha^2)^{-1} V^\frac12},
	\quad V^\frac12 \coloneqq |V|^\frac12 \sgn(V),
	\quad \alpha > 0,
\end{align}
is a well-defined Hilbert-Schmidt operator in $ L^2(\R^2) $, see Proposition~\ref{Prop:HS} below. Following
Kato's construction in \cite{Kato-1966-162} we set $ R_0(\alpha) \coloneqq (-\Delta+\alpha^2)^{-1} $, $ \alpha > 0 $,
and 
\begin{equation}\label{H_res}
\begin{aligned} 
	R(\alpha)
	=
	R_0(\alpha) + \overline{R_0(\alpha) V^{\frac12}} [I - Q(\alpha)]^{-1} |V|^\frac12 R_0(\alpha),& \\
	\alpha\in\lbrace \beta > 0 : 1 \in \rho(Q(\beta)) \rbrace,&
\end{aligned}
\end{equation}
and conclude in Proposition~\ref{Prop:H} below that $ R(\alpha) $ is the resolvent of a semibounded 
self-adjoint operator $H$ (which is an extension of the symmetric operator $ {-\Delta - V} $ defined on
$ \dom(-\Delta - V)= H^2(\R^2) \cap \dom V $) in $ L^2(\R^2) $.

As a first step we analyze the Birman-Schwinger operator $ Q(\alpha) $ given by \eqref{Q}.
Recall that for $ \alpha > 0 $ the resolvent $ {(-\Delta + \alpha^2)^{-1}} $ 
of the free Laplacian in $ L^2(\R^2) $
is an integral operator with kernel
\begin{equation} \label{res_schroed}
\begin{aligned}
	\calG(x,y;\alpha)
		=
		\frac{1}{2\pi} K_0(\alpha|x-y|),
	\quad x,y \in \R^2, \quad x \neq y,
\end{aligned}
\end{equation}
where $ K_0 : \R^+ \to \R^+ $ denotes the modified Bessel function of the
second kind of order zero (see
\cite[Chap.~9.6]{AbramowitzStegun} for more details). Recall that $ K_0 $ is monotonically decreasing and has the expansions
\begin{align}   
    K_0(w) &= -\ln w + \ln 2 - \gamma + \calO( w^2 \ln w ),  \quad w \to 0+,  \label{K0_series} \\
    K_0(w) &= \left( \frac{\pi}{2w} \right)^{\frac12} e^{-w} \left( 1 + \calO(w^{-1}) \right),
    \quad w \to +\infty, \label{K0_expansion}
\end{align}
where $ \gamma $ is the Euler-Mascheroni constant, see \cite[Eq.~(9.6.13)~and~(9.7.2)]{AbramowitzStegun}, respectively.

In the next lemma we collect useful estimates for the Green's function \eqref{res_schroed}.

\begin{lemma}  \label{Lem:ineq}
Let $ \calG(x,y;\alpha) $ be given by \eqref{res_schroed}.
Then there exists a constant $C>0$ such that for all $ \alpha \in (0, \frac1e)$, all $ s \in [0,2] $ and all $x,y \in \R^2$ with $x \neq y$ the following inequalities hold.
\begin{enumerate}
\item $ \left| \calG(x,y;\alpha) + \frac{\ln\alpha}{2\pi} \right|^s \leq C(1 + |{\ln|x-y|}|^s) $.
\item $ \left| \calG(x,y;\alpha) + \frac{\ln\alpha}{2\pi} \right|^s \leq C |{\ln\alpha}|^s $ if $ \alpha|x-y| \geq 1 $. 
\item $ \calG(x,y;\alpha)^2 \leq C ((\ln\alpha)^2 + (\ln|x-y|)^2) $.
\end{enumerate}
\end{lemma}

\begin{proof}
In the following $ C $ denotes a positive constant that may change between estimates. For ease of notation we set
\begin{align} \label{proof:def:w}
   w := \alpha|x-y|
\end{align}
for $ x, y \in \R^2 $ with $ x \neq y $; note that $ w > 0 $ and that the kernel we estimate reads
\begin{align}
    \calG(x,y;\alpha) + \frac{\ln\alpha}{2\pi} = \frac{K_0(w)+\ln\alpha}{2\pi}. \label{proof:weber_2D_1}
\end{align}
Note also that it suffices to prove both items (i) and (ii) for $ s = 1 $ since the general case then follows by taking
both sides to the power of $ s \in [0,2] $ and in the case of (i) applying the elementary inequality
$ |a+b|^s \leq  2(|a|^s + |b|^s) $.

\begin{enumerate}[\upshape (i), wide]
\item If $ w < 1 $, then \eqref{K0_series} gives us
\begin{align*}
    | K_0(w) + \ln w |
        \leq
    C \left( 1 + |w|^2 |{\ln w}| \right)
        \leq
    C
\end{align*}
and since $ \ln\alpha  = \ln w - \ln |x-y| $ this implies
\begin{align}\label{holladi}
    |K_0(w) + \ln\alpha|
        =
    |K_0(w) + \ln w - \ln|x-y| |
        \leq
    C(1 + |{\ln|x-y|}|).
\end{align}
If $ w \geq 1 $ we use that $ K_0 $ is monotonically decreasing and infer with $ \alpha\in(0,\frac1e) $
\begin{equation} \label{eq:K0+ln}
    |{K_0(w) + \ln\alpha}|
   		\leq
    |K_0(w)| + |{\ln\alpha}|
    		\leq
    |K_0(1)| + |{\ln\alpha}|
    		\leq
    	C|{\ln\alpha}|.
\end{equation}
Since $ e < \alpha^{-1} \leq |x-y| $ if $ w \geq 1 $ by \eqref{proof:def:w}, we conclude
\begin{equation*}
    |{K_0(w) + \ln\alpha}|
    \leq
    C|{\ln\alpha}|
    	=
    C\ln(\alpha^{-1})
    	\leq
    C \ln|x-y|,
\end{equation*}
which together with \eqref{holladi} shows (i).
%

\item By the definition \eqref{proof:def:w} of $ w $ we have that $ \alpha|x-y| \geq 1 $ 
is equivalent to $ w \geq 1 $ so 
the claim follows from \eqref{eq:K0+ln}. 

\item By applying the elementary inequality $ (a+b)^2 \leq 2a^2 + 2b^2 $ we find
\begin{align*}
	\calG(x,y;\alpha)^2 
	\leq
	2\bigg( \calG(x,y;\alpha) + \frac{\ln\alpha}{2\pi} \bigg)^2 + 2 \bigg( \frac{\ln\alpha}{2\pi} \bigg)^2 
\end{align*}
so the claim follows from (i) with $ s = 2 $.
\qedhere
\end{enumerate}
\end{proof}

As a consequence of Lemma~\ref{Lem:ineq} we obtain the compactness of the Birman-Schwinger operator $ Q $ as well
as the decay of $ \| Q(\alpha) \| $ as $ \alpha\to+\infty $.

\begin{proposition} \label{Prop:HS}
Assume that $ V \in L^1(\R^2) $ is real-valued and satisfies \eqref{V_roll}. 
Then for any $ \alpha > 0 $ the Birman-Schwinger operator
$ Q(\alpha) $ in \eqref{Q} 
is a Hilbert-Schmidt operator and one has $ \| Q(\alpha) \|_{\rm{HS}} \to 0 $ as $ \alpha\to+\infty $.
\end{proposition}

\begin{proof}
We start by proving that $ Q(\alpha) $ is a Hilbert-Schmidt operator for any $ \alpha > 0 $, that is, we verify 
\begin{align} \label{claim_HS}
	\| Q(\alpha) \|_{\rm{HS}}^2 = \int_{\R^4} |V(x)| \calG(x,y;\alpha)^2 |V(y)| \dmath(x,y) < \infty.
\end{align}

Recall first that $ K_0 : \R^+ \to \R^+ $ is monotonically decreasing. Thus, 
for fixed $ x,y\in\R^2 $, $ x \neq y $, we have by \eqref{res_schroed} that $ \calG(x,y;\alpha)^2 $ is monotonically
decreasing in $ \alpha $ and hence it suffices to prove the claim for $ \alpha\in (0,\frac1e) $.

Fix $ \alpha\in (0,\frac1e) $; we show \eqref{claim_HS} by splitting the integral into the two 
regions $ \alpha|x-y| < 1 $ and $ \alpha|x-y| \geq 1 $.
For the region $ \alpha|x-y| < 1 $ we apply Lemma~\ref{Lem:ineq}(iii) and obtain for some $ C > 0 $ the inequality
\begin{align} \label{G_in}
	\calG(x,y;\alpha)^2 &\leq C((\ln\alpha)^2 + (\ln|x-y|)^2),
	\quad \alpha|x-y| < 1.
\end{align}
For the region $ \alpha|x-y| \geq 1 $ we use \eqref{res_schroed} and the monotonicity of $ K_0 $ to infer
\begin{align} \label{G_out}
	\calG(x,y;\alpha)^2 \leq \bigg(\frac{K_0(1)}{2\pi}\bigg)^2,
	\quad \alpha|x-y| \geq 1.
\end{align}
By combining \eqref{G_in} and \eqref{G_out} we obtain for some $ C > 0 $
\begin{align*}
	\| Q(\alpha) \|_{\rm{HS}}^2
	\leq
	C\bigg( \| V \|_{L^1(\R^2)}^2 + \int_{\alpha|x-y|<1} |V(x)| (\ln|x-y|)^2 |V(y)| \dmath(x,y) \bigg).
\end{align*}
The latter integral is finite. This is easily seen after further splitting the region $ \alpha|x-y|<1 $ into
$ |x-y|<e $ and $ e \leq |x-y| < 1/\alpha $ (note that the latter region is non-empty since by assumption $ \alpha\in(0,\frac1e) $).
For $ {|x-y| < e} $ the finiteness now follows from Assumption \eqref{V_roll} 
and for $ e \leq |x-y| < 1/\alpha $ it follows from the boundedness of $ (\ln|x-y|)^2 $. 
In summary, \eqref{claim_HS} holds.

For the second claim, we justify below that dominated convergence yields
\begin{align} \label{claim:Q_zero}
	\int_{\R^4} |V(x)|\calG(x,y;\alpha)^2|V(y)| \dmath(x,y) \to 0,
	\quad{\alpha\to +\infty}.
\end{align}
To this end, \eqref{K0_expansion} implies
$ K_0(\alpha|x-y|) \to 0 $ and by \eqref{res_schroed} also $ \calG(x,y;\alpha) \to 0 $ as $ \alpha\to +\infty $ for
every $ x \neq y $ and hence for
almost every $ (x,y) \in \R^4 $. Furthermore, by the monotonicity of $ K_0 $ on $ \R^+ $ we have for any 
$ \alpha_0 \in (0,\frac1e) $ and all $ \alpha\geq\alpha_0 $
\begin{align*}
	|V(x)|\calG(x,y;\alpha)^2|V(y)| \leq |V(x)|\calG(x,y;\alpha_0)^2|V(y)|,
	\quad x,y \in \R^2, \quad x \neq y,
\end{align*}
which by the first part of this proof is an integrable upper bound.
\end{proof}

It turns out that the Birman-Schwinger operator being Hilbert-Schmidt also ensures that $ V $ is a relatively form compact
perturbation of $ -\Delta $ in $ L^2(\R^2) $, that is, 
\begin{equation}\label{rfc}
\dom (-\Delta+1)^{\frac12}=H^1(\R^2)\subset\dom |V|^\frac12 
 \quad\text{and}\quad
|V|^\frac12 (-\Delta+1)^{-\frac12} 
\end{equation}
is a compact operator in $ L^2(\R^2) $.

\begin{corollary} \label{cor:V_formc}
Assume that $ V \in L^1(\R^2) $ is real-valued and satisfies \eqref{V_roll}. Then $ V $ is 
relatively form compact with respect to $ -\Delta $ in $ L^2(\R^2) $.
\end{corollary} 

\begin{proof}
We show first that $ A\coloneqq |V|^\frac12 (-\Delta+1)^{-\frac12} $ is bounded and everywhere defined.
Consider for $ n \in \N $ the operators $ A_n $ in $ L^2(\R^2) $ defined by
\begin{align*}
	A_n \coloneqq V_n (-\Delta + 1)^{-\frac12},
	\quad 
	V_n(x) \coloneqq \min\lbrace |V(x)|^\frac12, n \rbrace.
\end{align*}
Clearly, $ V_n \in L^2(\R^2) \cap L^\infty(\R^2) $, and hence each $ A_n $, $n \in \N$, is bounded and everywhere defined. 
A short computation taking adjoints shows
\begin{align*}
	\| A_n^* f \|_{L^2(\R^2)}^2=\| (-\Delta + 1)^{-\frac12} V_n f\|_{L^2(\R^2)}^2
	\leq
	\| V_n (-\Delta + 1)^{-1} V_n \|_{\rm{HS}} \| f \|_{L^2(\R^2)}^2
\end{align*}
for any $ f \in L^2(\R^2) $. Using $ |V_n(x)|^2 \leq |V(x)| $ for all $ x \in \R^2 $ 
we obtain by the monotonicity of the Hilbert-Schmidt norms and Proposition~\ref{Prop:HS} that
\begin{align*}
\|A_n\|^2 = \|A_n^*\|^2 \leq \|Q(1)\|_{\rm{HS}} < \infty,
	\quad n \in \N.
\end{align*}
Since $ V_n(x)^2 \uparrow |V(x)|$ as $ n \to \infty $ monotone convergence implies 
\begin{align*}
\int_{\R^2} |V(x)| |(-\Delta+1)^{-\frac12} f(x)|^2 \dmath x
& = \lim_{n \to \infty}
\int_{\R^2} V_n(x)^2 |(-\Delta+1)^{-\frac12} f(x)|^2 \dmath x
\\
& \leq \limsup_{n \to \infty} \|A_n\|^2 \| f \|_{L^2(\R^2)}^2 
\\
& 
\leq \|Q(1) \|_{\rm{HS}} \| f \|_{L^2(\R^2)}^2,
\end{align*}
for any $ f \in L^2(\R^2) $. Thus we conclude the first inclusion in \eqref{rfc} and that 
$ A$ is bounded and everywhere defined in $ L^2(\R^2) $.
This also  implies
\begin{align*}
	A A^* = |V|^\frac12 (-\Delta+1)^{-\frac12} \overline{(-\Delta+1)^{-\frac12} |V|^\frac12} = Q(1)\sgn(V);
\end{align*}
since $ Q(1) $ is compact (see Proposition~\ref{Prop:HS}) and $ \sgn(V) $ is bounded, it follows that
$ A $ and $ A^* $ are compact as well; the second claim in \eqref{rfc} follows.
\end{proof}

We now employ the compactness of the Birman-Schwinger operator to show that $ R(\alpha) $ given by 
\eqref{H_res} defines the resolvent of a self-adjoint operator.

\begin{proposition} \label{Prop:H}
Assume that $ V \in L^1(\R^2) $ is real-valued and satisfies \eqref{V_roll}. 
Then the operator $ R(\alpha) $ given by \eqref{H_res} 
for $ \alpha\in\lbrace \beta > 0 : 1 \in \rho(Q(\beta)) \rbrace $ 
defines a self-adjoint operator $ H $ in $ L^2(\R^2) $ by
\begin{align} \label{H}
	R(\alpha) = (H + \alpha^2)^{-1},
\end{align}
which has the more explicit form
\begin{align} \label{H_form}
	H = -\Delta - V, 
	\quad \dom H = \bigl\{ f \in H^1(\R^2) : (-\Delta - V) f \in L^2(\R^2) \bigr\}. 
\end{align}
Moreover, $ H $ is bounded from below, $ \sigma_{\rm{ess}}(H) = [0,\infty) $, and one has
\begin{align} \label{BS}
	\dim\ker(H+\alpha^2) = \dim\ker(I-Q(\alpha)),
	\quad \alpha>0.
\end{align}
Finally, $ H $ is a self-adjoint extension of the symmetric operator $ {-\Delta - V} $ defined on $\dom(-\Delta - V) = {H^2(\R^2)\cap\dom V} $
in $ L^2(\R^2) $.
\end{proposition}


\begin{proof}[Proof of Proposition~\ref{Prop:H}]
The statement essentially follows from well-known results by Kato~\cite{Kato-1966-162} and Konno and Kuroda \cite{KonnoKuroda}.
However, we follow the presentation in \cite{GesztesyLatushkinMitreaZinchenko2005}
and apply the results therein with $ {\calH = \calK = L^2(\R^2)} $,
$ H_0 = -\Delta $, $ \dom H_0 = H^2(\R^2) $, $ A = -\smash{|V|^\frac12} $ and $ B = \smash{V^\frac12} $,
where $ A $ and $ B $ are understood as maximal multiplication operators in $ L^2(\R^2) $.

More precisely, since by assumption $ V \in L^1(\R^2) $, we have $ \smash{|V|^{\frac12}} \in L^2(\R^2) $ and as a consequence, the
Sobolev embedding $ H^2(\R^2) \hookrightarrow L^\infty(\R^2) $ implies 
$ {\dom H_0 \subset \smash{\dom|V|^\frac12}} $ 
(actually, by \eqref{rfc} we even know $ {H^1(\R^2) \subset \smash{\dom|V|^\frac12}} $). Next, 
Proposition~\ref{Prop:HS} yields the compactness of the Birman-Schwinger operator $ Q(\alpha) $ 
and $ \lbrace \beta > 0 : 1 \in \rho(Q(\beta)) \rbrace \neq \emptyset $.
Hence, \cite[Thm.~2.3]{GesztesyLatushkinMitreaZinchenko2005} implies that $ H $ is 
a densely defined and closed extension of the operator $ {-\Delta - V} $ defined on $\dom(-\Delta - V) = {H^2(\R^2)\cap\dom V} $ in $ L^2(\R^2) $
and \cite[Thm.~3.2]{GesztesyLatushkinMitreaZinchenko2005} yields the Birman-Schwinger principle \eqref{BS}.
Since $ V $ is real-valued, $ H $ is symmetric. Moreover, $ \| Q(\alpha) \| \to 0 $ as $ \alpha\to +\infty $, so
$ -\alpha^2 \in \rho(H) $ for all sufficiently large $ \alpha $. Hence $ H $ is self-adjoint and bounded from below in $ L^2(\R^2) $.
The stability of the essential spectrum follows from 
\cite[Thm.~4.5]{GesztesyLatushkinMitreaZinchenko2005}.

Finally, $ H $ having the explicit form \eqref{H_form} follows from the relative form compactness of $V$ in
Corollary~\ref{cor:V_formc}.
Indeed, 
the operator $ H $ defined by \eqref{H} coincides with the self-adjoint realization \eqref{H_form}
of $ {-\Delta - V} $ in $ L^2(\R^2) $ constructed via standard form methods, see, e.g.,  
\cite[Chap.~VI.2~and~VI.3]{KatoPT}, \cite[Chap.~VIII.6]{ReedSimonI}, and \cite[Lemma~3.8]{BehrndtLangerName}.
\end{proof}

\section{Weak coupling regime and the proof of Theorem~\ref{Thm}} \label{sec:WC}
In the following we are interested in the negative eigenvalues of the operators $ H_\varepsilon $ formally given by
$ -\Delta - \varepsilon V $ in $ L^2(\R^2) $ and defined rigorously below using the construction from the previous section
(with $V$ replaced by $\varepsilon V$), and we prove our main result Theorem~\ref{Thm}.

Assume that $ \varepsilon > 0 $ is small and that
$ V \in L^1(\R^2) $ is real-valued and satisfies
\eqref{V_roll}. Following the construction in Section~\ref{sec:def},
set $ {R_0(\alpha) \coloneqq (-\Delta+\alpha^2)^{-1}} $ and define 
\begin{equation}\label{H_res_eps}
\begin{aligned} 
	R_\varepsilon(\alpha)
	=
	R_0(\alpha) + \varepsilon \overline{R_0(\alpha) V^{\frac12}} [I - \varepsilon Q(\alpha)]^{-1} |V|^\frac12 R_0(\alpha),& \\
	\alpha\in\lbrace \beta > 0 : 1 \in \rho(\varepsilon Q(\beta)) \rbrace,&
\end{aligned}
\end{equation}
where $ Q(\alpha) $ is the compact Birman-Schwinger operator in $ L^2(\R^2) $
given by \eqref{Q}. It follows as in Proposition~\ref{Prop:H} that 
the operator $ R_\varepsilon(\alpha) $ for $ \alpha\in\lbrace \beta > 0 : 1 \in \rho(\varepsilon Q(\beta)) \rbrace $
defines a self-adjoint operator $ H_\varepsilon $ in $ L^2(\R^2) $ by 
\begin{align}\label{hierister} 
	R_\varepsilon(\alpha) = (H_\varepsilon + \alpha^2)^{-1},
\end{align}
which has the more explicit form
\begin{align} \label{H_eps_form}
	H_\varepsilon = -\Delta - \varepsilon V, 
	\quad \dom H_\varepsilon = \bigl\{ f \in H^1(\R^2) : (-\Delta - \varepsilon V) f \in L^2(\R^2) \bigr\}.
\end{align}
Moreover, $ H_\varepsilon $ is bounded from below, $ \sigma_{\rm{ess}}(H_\varepsilon) = [0,\infty) $, one has
\begin{align} \label{BS_eps}
	\dim\ker(H_\varepsilon + \alpha^2) = \dim\ker(I-\varepsilon Q(\alpha)),
	\quad \alpha>0,
\end{align}
and $ H_\varepsilon $ is a self-adjoint extension of the symmetric operator
$ {-\Delta -\varepsilon V} $ defined on $\dom(-\Delta -\varepsilon V) = {H^2(\R^2)\cap\dom V} $ in $ L^2(\R^2) $.

\subsection{Preparatory estimates}
To analyze the eigenvalues of $ H_\varepsilon $ in the weak coupling limit
Simon \cite{Simon1976} decomposed the Birman-Schwinger operator $ Q(\alpha) $ as
\begin{align} \label{Q=L+M_schroed}
	Q(\alpha) = L(\alpha) + M(\alpha),\quad \alpha>0;
\end{align}
with integral operators $ L(\alpha) $ and $ M(\alpha) $ in $ L^2(\R^2) $ that have the kernels
\begin{align}
	\calL(x,y; \alpha) &:= |V(x)|^{\frac{1}{2}} g(\alpha) V(y)^{\frac{1}{2}}, \label{L_kernel} \\
	\calM(x,y; \alpha) &:= |V(x)|^{\frac{1}{2}} \left( \calG(x,y;\alpha) - g(\alpha) \right) V(y)^{\frac{1}{2}}, \label{M_kernel}
\end{align}
respectively, and the function $ g $ is given by
\begin{equation} \label{h}
\begin{aligned}
	g(\alpha) = -(2\pi)^{-1} \ln\alpha.
\end{aligned}
\end{equation}
It is useful to notice that $ L $ is of rank one and that its norm is given by
\begin{align*}
	\| L(\alpha) \| = |g(\alpha)|\| V \|_{L^1(\R^2)},
\end{align*}
and hence $ \| L(\alpha) \| $ has a logarithmic singularity at $ \alpha = 0 $.
Meanwhile, the singularity of $ \| M(\alpha) \| $ is weaker, which is the content of the next lemma.

\begin{lemma} \label{Lem:M}
Assume that $ V \in L^1(\R^2) $ is real-valued and satisfies \eqref{V_roll} 
and \eqref{V_ln_s} for some $ s \in [0,1) $. Then 
\begin{equation}\label{roll_exists}
 \int_{\R^4} |V(x)||{\ln|x-y|}|^s|V(y)|\dmath(x,y) < \infty
\end{equation}
and for the integral operator $ M(\alpha) $ in $ L^2(\R^2)$ for $\alpha\to 0+$ one has 
\begin{enumerate}
	\item $ \| M(\alpha) \|^2 = o(|g(\alpha)|^{2-s}) $,
	\item $ |(M(\alpha)|V|^\frac12, V^\frac12)| = o(|g(\alpha)|^{1-s}) $.
\end{enumerate}
\end{lemma}

\begin{proof}
We start with some preliminary observations and verify \eqref{roll_exists}. Note  
that the 
inequality
\begin{align} \label{ineq:ln_r}
	|{\ln|x-y|}|^r \leq 1 + (\ln|x-y|)^2,
	\quad x,y \in \R^2, \quad x \neq y,\quad r\in[0,2],
\end{align}
holds. To check \eqref{roll_exists}, observe that the existence of the integral over the region $ |x-y| < e $ follows from \eqref{ineq:ln_r} applied
with $ r = s \in [0,1) $ and our assumptions $ V \in L^1(\R^2) $ 
and \eqref{V_roll}.
For the region $ |x-y| \geq e $ we make use of the elementary inequality
\begin{align} \label{ineq:ln_tmp}
	\ln(1+|x|) \leq \ln 2 + \mathbbm{1}_{\lbrace |x| \geq 1\rbrace}(x) \ln|x|.
\end{align}
Since $ e \leq |x-y| \leq |x|+|y| \leq (1+|x|)(1+|y|) $ we have by \eqref{ineq:ln_tmp}
\begin{align*}
	(\ln|x-y|)^s
	&\leq 
	(\ln(1+|x|)+\ln(1+|y|))^s   \\
	&\leq
	\big( 2\ln 2 + \mathbbm{1}_{\lbrace |x| \geq 1\rbrace}(x) \ln|x|
		 + \mathbbm{1}_{\lbrace |y| \geq 1\rbrace}(y) \ln|y| \big)^s \\
	&\leq
	C \big( 1 + \mathbbm{1}_{\lbrace |x| \geq 1\rbrace}(x) |{\ln|x|}|^s
		 + \mathbbm{1}_{\lbrace |y| \geq 1\rbrace}(y) |{\ln|y|}|^s \big);
\end{align*}
this together with $ V \in L^1(\R^2) $ and \eqref{V_ln_s} shows the finiteness of the integral
over the region $ |x-y| \geq e $. Summing up, we have shown \eqref{roll_exists}.

For the proofs of assertions (i) and (ii) 
we again use the convention that $ C $ denotes a positive constant that may change between estimates.
By the definition \eqref{M_kernel} of the kernel of $ M $ we have
\begin{align*}
	\| M(\alpha) \|^2
	&\leq 
	\int_{\R^4} |V(x)| |\calG(x,y;\alpha) - g(\alpha)|^2 |V(y)| \dmath(x,y),
	\\
	|(M(\alpha)|V|^\frac12, V^\frac12)|
	&\leq
	\int_{\R^4} |V(x)| |\calG(x,y;\alpha) - g(\alpha)| |V(y)| \dmath(x,y).
\end{align*}
Hence, employing \eqref{h}, both
assertions follow if we show for $ t \in \lbrace 1,2 \rbrace $ that
\begin{align} \label{claim_t}
	\int_{\R^4} |V(x)| |\calG(x,y;\alpha) - g(\alpha)|^t |V(y)| \dmath(x,y)
	=
	o(|{\ln\alpha}|^{t-s})
\end{align}
as $ \alpha\to 0+ $.
In the following we assume $ \alpha\in(0,\frac1e) $ so that
the inequalities in Lemma~\ref{Lem:ineq} are valid. In particular, by Lemma~\ref{Lem:ineq}(i), there exists $C>0$ such that for all 
$\alpha \in (0, \frac 1e)$, we have
\begin{align} \label{ineq:W}
	|\calG(x,y;\alpha) - g(\alpha)|^t \leq C\left( 1 + |{\ln|x-y|}|^t \right),
	\quad (x,y) \in \R^4, \quad x \neq y.
\end{align}
Next, introduce the sets
\begin{equation} \label{Omega}
\begin{aligned}
	\Omega_1 = \Omega_1(\alpha) &= \big\lbrace (x,y) \in \R^4 : \ln|x-y| < 1 \big\rbrace, \\
	\Omega_2(\alpha) &= \big\lbrace (x,y) \in \R^4 : 1 \leq \ln|x-y| < |{\ln\alpha}|^{\frac12} \big\rbrace, \\
	\Omega_3(\alpha) &= \big\lbrace (x,y) \in \R^4 : |{\ln\alpha}|^{\frac12} \leq \ln|x-y| < |{\ln\alpha}| \big\rbrace, \\
	\Omega_4(\alpha) &= \big\lbrace (x,y) \in \R^4 : |{\ln\alpha}| \leq \ln|x-y| \big\rbrace.
\end{aligned}
\end{equation}
Note that $ \R^4 = \bigcup_{k=1}^4 \Omega_k(\alpha) $ as a disjoint union and that
$ \Omega_k(\alpha) \neq \emptyset $ for each $ k \in \lbrace 1, \ldots, 4 \rbrace $ since by assumption
$ \alpha\in(0,\frac1e) $, and thus $\smash{1< |{\ln\alpha}|^{\frac12}< |{\ln\alpha}|}$.
Consequently, \eqref{claim_t} follows if we consider
for $ t \in \lbrace 1,2 \rbrace $ and $ k \in \lbrace 1, \ldots, 4 \rbrace $ the integrals
\begin{align} \label{I_k}
	I_k(\alpha) \coloneqq \int_{\Omega_k(\alpha)} |V(x)| |\calG(x,y;\alpha) - g(\alpha)|^t |V(y)| \dmath(x,y),
\end{align}
and show for each $ k \in \lbrace 1, \ldots, 4 \rbrace $ that 
\begin{align} \label{claim:I_k}
	I_k(\alpha) = o(|{\ln\alpha}|^{t-s}), \quad {\alpha\to 0+}.
\end{align}

It is clear that \eqref{claim:I_k} holds for $ I_1(\alpha) $ since by \eqref{ineq:W}
\begin{align*} 
	I_1(\alpha)
		\leq
	C \bigg( \| V \|_{L^1(\R^2)}^2 + \int_{|x-y|<e} |V(x)| |{\ln|x-y|}|^t |V(y)| \dmath(x,y) \bigg) < \infty;
\end{align*}
the last integral is finite by \eqref{ineq:ln_r} applied with $ r = t $ 
and our assumptions $ V \in L^1(\R^2) $ and \eqref{V_roll}.
In particular, $ I_1(\alpha) $ is bounded for $ \alpha\in(0,\frac1e) $ and since 
$ |{\ln\alpha}|^{t-s}\to \infty $ as $ \alpha\to 0+ $ the claim \eqref{claim:I_k} follows for $k=1$.

We continue with the integral 
$ I_2(\alpha) $. Note that for any $ (x,y) \in \Omega_2(\alpha) $
\begin{equation} \label{ineq:Om_2}
\begin{aligned} 
	1 \leq (\ln|x-y|)^t 
	&=
	({\ln|x-y|})^{t-s} ({\ln|x-y|})^s \\
	&<
	|{\ln\alpha}|^{\frac{t-s}{2}} ({\ln|x-y|})^s.
\end{aligned}
\end{equation}
By employing this inequality in \eqref{ineq:W} we infer
\begin{equation*}  
\begin{aligned} 
	|\calG(x,y;\alpha) - g(\alpha)|^t
		\leq
	C(1 + |{\ln\alpha}|^{\frac{t-s}{2}} ({\ln|x-y|})^s	)
		\leq
	C|{\ln\alpha}|^{\frac{t-s}{2}} ({\ln|x-y|})^s		
\end{aligned}
\end{equation*}
and hence by \eqref{roll_exists}
\begin{align*}
	I_2(\alpha) 
	\leq
	C|{\ln\alpha}|^{\frac{t-s}{2}} .
\end{align*}
In particular, since by assumption $ s < 1 $ and $ t \in \lbrace 1, 2 \rbrace $, we conclude
\begin{align*}
	\frac{I_2(\alpha)}{|{\ln\alpha}|^{t-s}} \leq
	C |{\ln\alpha}|^{\frac{s-t}{2}} 
	\to
	0,
	\quad \alpha\to 0+,
\end{align*}
i.e., \eqref{claim:I_k} holds for $k=2$.

Next we show \eqref{claim:I_k} for $ I_3(\alpha) $ and $ I_4(\alpha) $. A similar estimate
as in \eqref{ineq:Om_2} implies for any $ (x,y)\in\Omega_3(\alpha) $ 
\begin{align*} 
	(\ln|x-y|)^t 
	\leq
	|{\ln\alpha}|^{t-s} (\ln|x-y|)^s.
\end{align*}
Using the inequality in \eqref{ineq:W} implies that 
\begin{equation}\label{ineq:W_I3}
\begin{aligned} 
	|\calG(x,y;\alpha) - g(\alpha)|^t
		&\leq 
	C( 1 + |{\ln\alpha}|^{t-s} (\ln|x-y|)^s ) \\
		&\leq
	C|{\ln\alpha}|^{t-s} (\ln|x-y|)^s,
	\quad (x,y) \in \Omega_3(\alpha).
\end{aligned}
\end{equation}
For $ (x,y) \in \Omega_4(\alpha) $ we see (recall that $ \alpha\in(0,\frac1e) $)
\begin{align*}
	|x-y| \geq e^{|{\ln\alpha}|} = e^{-\ln\alpha} = \frac{1}{\alpha}
\end{align*}
so we can employ Lemma~\ref{Lem:ineq}(ii) and (i), respectively, to obtain
\begin{equation} \label{ineq:W_I4}
\begin{aligned} 
	|\calG(x,y;\alpha) - g(\alpha)|^t
	&=
	|\calG(x,y;\alpha) - g(\alpha)|^{t-s} |\calG(x,y;\alpha) - g(\alpha)|^s \\
	&\leq 
	C|{\ln\alpha}|^{t-s} (1 + (\ln|x-y|)^s) \\
	&\leq
	C|{\ln\alpha}|^{t-s} (\ln|x-y|)^s,
	\quad (x,y) \in \Omega_4(\alpha).
\end{aligned}
\end{equation}
Finally, by combining \eqref{ineq:W_I3} with \eqref{ineq:W_I4} we obtain
\begin{align*}
	I_3(\alpha) + I_4(\alpha)
	\leq
	C|{\ln\alpha}|^{t-s} \int_{\Omega_3(\alpha) \cup \Omega_4(\alpha)} |V(x)|(\ln|x-y|)^s|V(y)| \dmath(x,y).
\end{align*}
By \eqref{roll_exists} and the definition of $ \Omega_3(\alpha) $ and $ \Omega_4(\alpha) $
dominated convergence implies that the last integral tends to zero as $ \alpha \to 0+ $ so
\eqref{claim:I_k} for $ k \in \lbrace 3,4 \rbrace $ follows. 
\end{proof}

\subsection{Negative eigenvalues of \texorpdfstring{$ H_\varepsilon $}{the Schrödinger operator} }

In the following we use the decomposition \eqref{Q=L+M_schroed} to characterize the negative eigenvalues of $ H_\varepsilon $ 
in certain subsets as zeros of a function $\Lambda_\varepsilon$.

\begin{proposition} \label{Prop:Lambda_schroed}
Let $ \varepsilon > 0 $, assume that $ V \in L^1(\R^2) $ is real-valued and satisfies \eqref{V_roll}, let
$ H_\varepsilon = -\Delta - \varepsilon V $  be defined as in \eqref{H_eps_form},
and let $ g $ be given by \eqref{h}. Then for any $ \alpha > 0 $ such that $ \| \varepsilon M(\alpha) \| < 1 $ we have that
$ -\alpha^2 $ is an eigenvalue of $ H_\varepsilon $ if and only if 
\begin{align} \label{Lambda_b}
	\Lambda_\varepsilon(\alpha) \coloneqq 1 - \varepsilon g(\alpha) ( [I-\varepsilon M(\alpha)]^{-1} |V|^\frac12, V^\frac12)
	=
	0.
\end{align}
Moreover, any eigenvalue of $ H_\varepsilon $ of this form is simple.
\end{proposition}

\begin{proof}
If $ \alpha > 0 $ is such that $ \| \varepsilon M(\alpha) \| < 1 $ we make use of
\eqref{Q=L+M_schroed} and 
the bounded invertibility of $ I - \varepsilon M(\alpha) $ and obtain
\begin{align} \label{fac:Q}
	I - \varepsilon Q(\alpha)
	=
	(I -\varepsilon M(\alpha)) \left( I - [I - \varepsilon M(\alpha)]^{-1} \varepsilon L(\alpha) \right).
\end{align}
With the help of the Birman-Schwinger principle \eqref{BS_eps} we then conclude 
\begin{align*}
	\dim\ker (H_\varepsilon + \alpha^2)
	&=
	\dim\ker (I - \varepsilon Q(\alpha) ) \\
	&=
	\dim\ker \left( I - [I - \varepsilon M(\alpha)]^{-1} \varepsilon L(\alpha) \right).
\end{align*}
Next, note that \eqref{L_kernel} implies $ L(\alpha) = g(\alpha) L_2 L_1 $ with 
\begin{align} \label{L1L2_def}
	L_1 : \begin{cases}
		L^2(\R^2) \to \C, \\
		f \mapsto (f, V^\frac12),
	\end{cases}
	\quad
	L_2 : \begin{cases}
		\C \to L^2(\R^2), \\
		\varphi\mapsto \varphi|V|^\frac12.
	\end{cases}
\end{align}
Recall that for two bounded and everywhere defined operators $ A : \calH \to \calG $ and $ B : \calG \to \calH $ one has
$ \sigma_p(AB) \cup \lbrace 0 \rbrace = \sigma_p(BA) \cup\lbrace 0 \rbrace $ and
\begin{align*}
	\dim\ker(I_\calG - AB) = \dim\ker(I_\calH - BA),
\end{align*}
see, e.g., \cite[Chap.~VII.4, Prob.~7]{ConwayFunc}.
In particular, using the factorization \eqref{L1L2_def} of $ L $ we conclude that
\begin{align*}
	\dim\ker \left( I - [I - \varepsilon M(\alpha)]^{-1} \varepsilon L(\alpha) \right)
	=
	\dim\ker \left( 1 - \varepsilon g(\alpha) L_1 [I - \varepsilon M(\alpha)]^{-1} L_2 \right).
\end{align*}
Since the definition of $ L_1 $ and $ L_2 $ implies
\begin{align*}
	1 - \varepsilon g(\alpha) L_1 [I - \varepsilon M(\alpha)]^{-1} L_2
	=
	1 - \varepsilon g(\alpha) ( [I - \varepsilon M(\alpha)]^{-1} |V|^\frac12, V^{\frac12} )
	=
	\Lambda_\varepsilon(\alpha)
\end{align*}
we have shown \eqref{Lambda_b}.

Finally, $ \dim \C = 1 $ so either $ \dim\ker\Lambda_\varepsilon(\alpha) = 1 $ if $ \Lambda_\varepsilon(\alpha) = 0 $ 
or alternatively $ \dim\ker\Lambda_\varepsilon(\alpha) = 0 $ if $ \Lambda_\varepsilon(\alpha) \neq 0 $. Hence any eigenvalue corresponding
to a zero of $ \Lambda_\varepsilon $ is simple.
\end{proof}

\subsection{Proof of Theorem~\ref{Thm}}
From the previous considerations it is clear that 
the operator $ {H_\varepsilon = -\Delta - \varepsilon V} $ in \eqref{H_eps_form}
is self-adjoint in $ L^2(\R^2) $, bounded from below, and one has $ \sigma_{\rm{ess}}(H_\varepsilon) = [0,\infty) $.
Thus it remains to verify (ii).
By Proposition~\ref{Prop:Lambda_schroed} it suffices to prove that
the function $ \Lambda_\varepsilon $ given in \eqref{Lambda_b} has a zero as $ \varepsilon\to 0+ $.
For ease of notation we set 
\begin{align*}
	U \coloneqq \int_{\R^2} V(x) \dmath x=(|V|^\frac12, V^\frac12)
\end{align*}
and note that by assumption $ U > 0 $. Let $ \alpha > 0 $ such that 
$ \| \varepsilon M(\alpha) \| < 1 $. 
By employing the expansion
\begin{align*}
	[I-\varepsilon M(\alpha)]^{-1} 
	= 
	I + \varepsilon M(\alpha) + \varepsilon^2 [I-\varepsilon M(\alpha)]^{-1} M(\alpha)^2
\end{align*}
in \eqref{Lambda_b} we find
\begin{equation} \label{Lambda_exp}
\begin{aligned} 
	\Lambda_\varepsilon(\alpha)
	=
	1 - g(\alpha) \big( U\varepsilon + r_\varepsilon(\alpha)),
\end{aligned}
\end{equation}
where we have defined
\begin{align} \label{error}
	r_\varepsilon(\alpha) \coloneqq
	\varepsilon^2 (M(\alpha)|V|^\frac12, V^\frac12) 
	+ \varepsilon^3 ( [I-\varepsilon M(\alpha)]^{-1} M(\alpha)^2 |V|^\frac12, V^\frac12 ).
\end{align}

Next, we make the substitution
\begin{align} \label{g_subs}
	\alpha(t) \coloneqq
	g^{-1} \left( 
	\frac{1}{U\varepsilon}(1+t) \right)
	=
	\exp\left( -\frac{2\pi}{U\varepsilon}(1+t) \right),
	\quad t \in \left[-\frac12, \frac12\right],
\end{align}
which is well-defined for any $ \varepsilon > 0 $ since
$ g : (0,1) \to (0,\infty) $ is bijective and by assumption $ U > 0 $.
Clearly, $ \sup_{t\in[-\frac12,\frac12]} |\alpha(t)| \to 0 $ as $ \varepsilon\to 0+ $ and
\begin{align}\label{g_alpha}
	\sup_{t\in[-\frac12,\frac12]}|g(\alpha(t))| \leq C\varepsilon^{-1}
\end{align}
for some $ C > 0 $ and all $ \varepsilon > 0 $. Hence, all the following asymptotics as $ \varepsilon\to 0+ $ hold
uniformly for all $ t\in[-\frac12,\frac12] $. By Lemma~\ref{Lem:M}(i) we have 
\begin{align} \label{est:M}
	\| \varepsilon M(\alpha(t)) \|^2
	=
	\varepsilon^2 \| M(\alpha(t)) \|^2
	=
	o( \varepsilon^{s} ),
	\quad {\varepsilon\to 0+}, 
\end{align}
i.e., $ \| \varepsilon M(\alpha(t)) \| \leq \frac12 $ for any $ t \in [-\frac12, \frac12] $ and all $ \varepsilon \in (0,\varepsilon_0) $
if $ \varepsilon_0 > 0 $ is chosen small enough. In particular,
$ {\Lambda_\varepsilon \circ \alpha : [-\frac12, \frac12] \to \R} $ is well-defined and
continuous for any $ \varepsilon \in (0,\varepsilon_0) $, which 
is seen by expanding $ [I-\varepsilon M(\alpha)]^{-1} $ in \eqref{Lambda_exp} into a Neumann series 
and using that the family $ M $ is real-analytic on $ (0,\infty) $ and hence
also on $ \alpha([-\frac12,\frac12]) \subset (0,\infty) $.
Note also that each zero $ t $ of $ {\Lambda_\varepsilon \circ \alpha} $
corresponds to a simple eigenvalue $ -\alpha(t)^2 $ of $ H_\varepsilon $, see Proposition~\ref{Prop:Lambda_schroed}.

We continue by proving the existence of a zero of $ \Lambda_\varepsilon \circ \alpha $.
First we derive an upper bound for $ r_\varepsilon \circ \alpha $. By \eqref{est:M} we have
\begin{align*}
	\varepsilon^3 \big| ( [I-\varepsilon M(\alpha(t))]^{-1} M(\alpha(t))^2 |V|^\frac12, V^\frac12 ) \big|
	&\leq
	\varepsilon^3 \frac{\| M(\alpha(t)) \|^2 \| V \|_{L^1(\R^2)}}{1 - \| \varepsilon M(\alpha(t)) \|}  \\
	&=
	o(\varepsilon^{1+s}),
	\quad {\varepsilon\to 0+},
\end{align*}
and with Lemma~\ref{Lem:M}(ii) and \eqref{g_alpha} we find 
\begin{align*}
	\varepsilon^2 \big| (M(\alpha(t))|V|^\frac12, V^\frac12) \big|
	=
	o(\varepsilon^{1+s}),
	\quad {\varepsilon\to 0+}.
\end{align*}
By combining the above estimates we see that the remainder $ r_\varepsilon $ in \eqref{error} satisfies
\begin{align*}
	|r_\varepsilon(\alpha(t))| = o(\varepsilon^{1+s}),
	\quad{\varepsilon\to 0+}.
\end{align*}
Hence, we conclude with \eqref{Lambda_exp} and \eqref{g_subs} that 
\begin{align*}
	(\Lambda_\varepsilon \circ \alpha)(t)
	=
	1 - g(\alpha(t)) \left( U\varepsilon + o(\varepsilon^{1+s}) \right)
	=
	-t + o(\varepsilon^s),
	\quad{\varepsilon\to 0+},
\end{align*}
which implies that $ \Lambda_\varepsilon \circ \alpha $ has a zero 
$ t_\varepsilon = o(\varepsilon^s) $ as $ \varepsilon\to 0+ $.

By Proposition~\ref{Prop:Lambda_schroed} it follows that 
$ \lambda_\varepsilon = - \alpha(t_\varepsilon)^2 $ is an eigenvalue of $ H_\varepsilon $ so 
by employing \eqref{g_subs} we find
\begin{align*}
	\ln(-\lambda_\varepsilon)
	=
	2 \ln( \alpha(t_\varepsilon))
	=
	- \frac{4\pi}{U\varepsilon} \left( 1 + t_\varepsilon \right)
	=
	- \frac{4\pi}{U\varepsilon} \left( 1 + o(\varepsilon^s) \right),
	\quad\varepsilon\to 0+,
\end{align*}
and all claims are shown. \qed

\section{Details on Example~\ref{Ex}(ii)} \label{sec:ex}

Here we check that the potential $ V_0 $ given by \eqref{Ex_0} satisfies \eqref{V_roll}, i.e.,
\begin{align} \label{claim:roll_2_ex}
	\int_{|x-y|<e} V_0(x) (\ln|x-y|)^2 V_0(y) \dmath(x,y) < \infty.
\end{align}
Recall that $ V_0 \in L^1(\R^2) $ and that we have 
\begin{align} \label{integrand_0}
	V_0(x)(\ln|x-y|)^2 V_0(y)
	=
	\frac{\mathbbm{1}_{\lbrace |x|<\frac13 \rbrace}(x)}{|x|^2 (\ln|x|)^4} (\ln|x-y|)^2
	\frac{\mathbbm{1}_{\lbrace |y|<\frac13 \rbrace}(y)}{|y|^2 (\ln|y|)^4}.
\end{align}
Since $ (\ln|x-y|)^2 $ is bounded for $ 1 \leq |x-y| < e $ it is clear that the integral \eqref{claim:roll_2_ex} is 
finite over the region $ 1 \leq |x-y| < e $.
Hence, it suffices to consider
the regions
\begin{align*}
	\Omega_1 &\coloneqq \left\lbrace (x,y)\in\R^4 : |x-y| < 1, \; |x-y| \geq \frac{|x|}{2} \right\rbrace, \\
	\Omega_2 &\coloneqq \left\lbrace (x,y)\in\R^4 : |x-y| < 1, \; |x-y| < \frac{|x|}{2} \right\rbrace.
\end{align*}

For $ (x,y) \in \Omega_1 $ we have 
\begin{align*}
	\mathbbm{1}_{\lbrace |x|<\frac13 \rbrace}(x)(\ln|x-y|)^2
	\leq 
	\mathbbm{1}_{\lbrace |x|<\frac13 \rbrace}(x) \left(\ln\left|\frac{x}{2}\right|\right)^2
	\leq 
	C \mathbbm{1}_{\lbrace |x|<\frac13 \rbrace}(x)(\ln|x|)^2
\end{align*}
and hence by \eqref{integrand_0}
\begin{align*}
	\int_{\Omega_1}
	V_0(x)&(\ln|x-y|)^2 V_0(y) \dmath(x,y) \\
	&\leq
	C \int_{\Omega_1}
	\frac{\mathbbm{1}_{\lbrace |x|<\frac13 \rbrace}(x)}{|x|^2 (\ln|x|)^2} 
	\frac{\mathbbm{1}_{\lbrace |y|<\frac13\rbrace}(y)}{|y|^2 (\ln|y|)^4} \dmath(x,y) \\
	&\leq
	C\bigg( \int_{|x|<\frac13} \frac{\dmath x}{|x|^2 (\ln|x|)^2} \bigg)
	\bigg( \int_{|y|<\frac13} \frac{\dmath y}{|y|^2 (\ln|y|)^4} \bigg) \\
	&< \infty.
\end{align*}

For $ (x,y) \in \Omega_2 $ the triangle inequality implies $ |x|/2 < |y| < 3|x|/2 $, and hence 
\begin{align*}
	\frac{\mathbbm{1}_{\lbrace |x|<\frac13\rbrace}(x)\mathbbm{1}_{\lbrace |y|<\frac13\rbrace}(y)}{|y|^2 (\ln|y|)^4}
	\leq
	C\frac{\mathbbm{1}_{\lbrace |x|<\frac13\rbrace}(x)\mathbbm{1}_{\lbrace |y|<\frac13\rbrace}(y)}{|x|^2 (\ln|x|)^4}.
\end{align*}
In particular, we find with \eqref{integrand_0}
\begin{align*}
	V_0(x)(\ln|x-y|)^2 V_0(y)
	\leq
	C\frac{\mathbbm{1}_{\lbrace |x|<\frac13 \rbrace}(x) \mathbbm{1}_{\lbrace |y|<\frac13\rbrace}(y)}{|x|^4 (\ln|x|)^8} (\ln|x-y|)^2.
\end{align*}
Using the substitution $ u(y) = x-y $ and Fubini's theorem we conclude
\begin{align*}
	\int_{\Omega_2} V_0(x)(\ln|x-y|)^2 V_0(y) \dmath(x,y)
	\leq
	C\int_{|x|<\frac13} \int_{|u|<\frac{|x|}{2}} \frac{(\ln|u|)^2 \dmath u \dmath x}{|x|^4 (\ln|x|)^8}.
\end{align*}
For the inner integral a simple estimate after integrating by parts twice shows
\begin{align*}
	\int_{|u|<\frac{|x|}{2}} (\ln|u|)^2 \dmath u
	=
	2\pi \int_0^{\frac{|x|}{2}} r (\ln r)^2 \dmath r 
	\leq
	C |x|^2 (\ln|x|)^2
\end{align*}
so we finally obtain
\begin{align*}
	\int_{|x|<\frac13} \int_{|u|<\frac{|x|}{2}} \frac{(\ln|u|)^2 \dmath u \dmath x}{|x|^4 (\ln|x|)^8}
	\leq
	C \int_{|x|<\frac13} \frac{\dmath x}{|x|^2 (\ln|x|)^6} < \infty.
\end{align*}
\\

\textbf{Acknowledgments.}
The authors thank the referees for a careful reading of the manuscript and various helpful suggestions to improve the text. 
This research was funded in part by the Austrian Science Fund (FWF) 10.55776/P 33568-N.

\end{document}